\pdfoutput=1
\RequirePackage{ifpdf}
\ifpdf 
\documentclass[pdftex]{sigma}
\else
\documentclass{sigma}
\fi

\newtheorem{Theorem}{Theorem}[section]
\newtheorem{Corollary}[Theorem]{Corollary}
\newtheorem{Lemma}[Theorem]{Lemma}
\newtheorem{Proposition}[Theorem]{Proposition}
{ \theoremstyle{definition}
\newtheorem{Definition}[Theorem]{Definition}
\newtheorem{Example}[Theorem]{Example}
\newtheorem{Remark}[Theorem]{Remark}
\newtheorem{Assumption}[Theorem]{Assumption}}

\numberwithin{equation}{section}

\def\N{{\mathbb N}}

\def\Q{{\mathbb Q}}

\def\C{{\mathbb C}}

\def\l{\left}
\def\r{\right}
\def\p{\prime}
\def\wtilde{\widetilde}
\def\what{\widehat}
\def\ul{\underline}
\def\ol{\overline}

\def\GL{{\rm GL}}

\def\Constr{{\rm Constr}}
\def\Hom{\mathop{\rm Hom}}
\def\End{\mathop{\rm End}}

\def\cE{{\mathcal E}}

\def\cM{{\mathcal M}}
\def\cN{{\mathcal N}}

\def\cL{{\mathcal{ L}}}

\def\g{{\mathfrak g}}

\def\veps{\varepsilon}

\def\a{\alpha}
\def\be{\beta}

\def\na{\nabla}

\def\liegal{\mathfrak g\mathfrak a\mathfrak l}

\def\Gal{{\rm Gal}_{\partial}}
\def\Aut{{\rm Aut}}
\def\constr{\mathfrak c\mathfrak o\mathfrak n\mathfrak s\mathfrak t\mathfrak r}
\def\gK{\g_{\rm Katz}}
\def\GK{G_{\rm Katz}}

\begin{document}
\allowdisplaybreaks

\newcommand{\arXivNumber}{1912.10567}

\renewcommand{\PaperNumber}{054}

\FirstPageHeading

\ShortArticleName{Reduced Forms of Linear Differential Systems and the Intrinsic Galois--Lie Algebra of Katz}

\ArticleName{Reduced Forms of Linear Differential Systems\\ and the Intrinsic Galois--Lie Algebra of Katz}

\Author{Moulay BARKATOU~$^\dag$, Thomas CLUZEAU~$^\dag$, Lucia DI VIZIO~$^\ddag$ and Jacques-Arthur WEIL~$^\dag$}

\AuthorNameForHeading{M.~Barkatou, T.~Cluzeau, L.~Di Vizio and J.-A.~Weil}

\Address{$^\dag$~XLIM, UMR7252, Universit\'e de Limoges et CNRS,\\
\hphantom{$^\dag$}~123 avenue Albert Thomas, 87060 Limoges Cedex, France}
\EmailD{\href{mailto:moulay.barkatou@unilim.fr}{moulay.barkatou@unilim.fr}, \href{mailto:thomas.cluzeau@unilim.fr}{thomas.cluzeau@unilim.fr},
\href{mailto:jacques-arthur.weil@unilim.fr}{jacques-arthur.weil@unilim.fr}}
\URLaddressD{\url{http://www.unilim.fr/pages_perso/moulay.barkatou/},\newline
\hspace*{13.5mm}\url{http://www.unilim.fr/pages_perso/thomas.cluzeau/},\newline
\hspace*{13.5mm}\url{http://www.unilim.fr/pages_perso/jacques-arthur.weil/}}

\Address{$^\ddag$~Universit\'e Paris-Saclay, UVSQ, CNRS, Laboratoire de math\'ematiques de Versailles,\\
\hphantom{$^\ddag$}~78000, Versailles, France}
\EmailD{\href{mailto:lucia.di.vizio@math.cnrs.fr}{lucia.di.vizio@math.cnrs.fr}}
\URLaddressD{\url{http://divizio.perso.math.cnrs.fr/}}

\ArticleDates{Received January 20, 2020, in final form June 04, 2020; Published online June 17, 2020}

\Abstract{Generalizing the main result of [Aparicio-Monforte A., Compoint E., Weil J.-A., \textit{J.~Pure Appl. Algebra} \textbf{217} (2013), 1504--1516], we prove that a linear differential system is in reduced form in the sense of Kolchin and Kovacic if and only if any differential module in an algebraic construction admits a constant basis. Then we derive an explicit version of this statement. We finally deduce some properties of the Lie algebra of Katz's intrinsic Galois group.}

\Keywords{linear differential systems; differential Galois theory; Lie algebras; reduced forms}

\Classification{34M03; 34M15; 34C20}

\section{Introduction}

Let us consider the field of rational functions $\C(x)$, with the derivation $\partial=\frac{{\rm d}}{{\rm d}x}$, and a linear differential system $\partial \vec y=A\vec y$, where $A$ is a square matrix of order $n$ with coefficients in~$\C(x)$. One can attach to such an object an algebraic group, called the differential Galois group, whose geometric properties encode the algebraic properties of the solutions of the linear differential system. The problem of calculating explicitly the differential Galois group of $\partial \vec y=A\vec y$ is old and still difficult. Among the several references, we cite \cite{AmMiPo18a,Compoint-Singer,Feng-hrushovski,Hrushcomp,Ho07b} that do not make any assumption on the order $n$ of the system. Implemented (or implementable) algorithms exist only for small dimensions $n$. See for example \cite{CoSa18a,Hessinger,Kovacic-algo,NgPu10a,Person,Singer-Ulmer-2nd3rdorder,HoRaUlWe99a,vanheoij-4thorder}.

Instead of calculating directly
the differential Galois group of $\partial \vec y=A\vec y$, one can try to study, or calculate, the Lie algebra of the differential Galois group, called {\em Galois--Lie algebra} in what follows. The Galois--Lie algebra already contains a significant part of the information.
Kolchin and Kovacic have proved that one can transform $\partial \vec y=A\vec y$ into an equivalent system $\partial \vec y=B\vec y$
defined over a finite extension $k$ of $\C(x)$, such that $B$ belongs to the set of $k$-rational points of the Galois--Lie algebra
(see \cite[Proposition 1.31]{vdPutSingerDifferential}). One can even prove that the Galois--Lie algebra is then ``generated'' by the entries of the matrix $B$. These ideas are formalized in \cite[Section~2.3]{ACW}.
The linear differential system $\partial \vec y=B\vec y$ is called {\em a reduced form} of $\partial \vec y=A\vec y$. A~linear differential system $\partial \vec y=A\vec y$ is said to be {\em in reduced form} if $A$ is a $k$-rational point of the Galois--Lie algebra.

{\bf Main results.}
In the present work, we prove three results on reduced forms and their relation to the Galois--Lie algebra.
First of all, we prove that a system is in reduced form if and only if any differential module in a construction admits a constant basis (see Theorem \ref{thm:main},
in particular the equivalence $1\Leftrightarrow 3$).
This extends the criterion for reduced form from \cite{ACW} which concerned only invariant lines (see the equivalence $1\Leftrightarrow 2$ in Theorem \ref{thm:main}).

Our second contribution is Theorem \ref{thm:main-explicit} that gives an effective characterization of a gauge transformation that transforms a
linear differential system into a reduced one. It may be consi\-dered as an effective counterpart of Theorem~\ref{thm:main}, based on the ``local''
data of the semi-invariants. Compared to the original result by Kolchin and Kovacic on the existence of reduced forms, we have to perform an algebraic extension
of the base field $k$,
which may not be optimal, in order to gain the effectivity.
Theorem \ref{thm:main-explicit}
extends a result
which only appeared in the course of the proof of \cite[Proposition 27]{ACW}, under the assumption that the differential system is completely
reducible.

Finally, we prove Theorem \ref{katz:algebra} on the Lie algebra of the intrinsic Galois group, introduced in~\cite{Katzbull}, where the idea of focusing on the Galois--Lie algebra rather than on the differential Galois group itself is pursued. Indeed, Katz introduces another Galois group for the linear differential system $\partial \vec y=A\vec y$, called the generic or the intrinsic Galois group. Then he considers the Lie algebra of such a group, for which he gives a conjectural description equivalent to a well-known conjecture of Grothendieck on the algebraicity of the solutions of a linear differential system. We will call it {\em the Katz algebra}. In the last section, we gather material from \cite{Andregrothendieck,Be92c,Katzbull,vdPutSingerDifferential} and show how our criteria for reduced forms, combined with standard Tannakian tools, clarify the structure of the Katz algebra. Namely, for a reductive group, Theorem \ref{katz:algebra} shows that the Katz algebra is a $k$-form of the Galois--Lie algebra.

{\bf The algorithm in \cite{barkatou2016computing}.}
In the latter reference, we showed how one can compute the Galois--Lie algebra of an (absolutely) irreducible linear differential system and hence (a good part of) its differential Galois group. Notice that in~\cite{DrWe19a}, it is shown how to derive the (connected) Galois group of the reduced form from the Galois--Lie algebra, so that one can actually recover the connected component of the Galois group of the original system.

The algorithm selects a Lie algebra that is potentially the Katz algebra and checks that there exists a gauge transformation that transforms its generators into
a set of constant generators of a Lie algebra, that is a candidate for being the Lie algebra of the Galois group.
To do so, it uses Theorem \ref{thm:main-explicit} in the particular case of a completely reducible differential system, hence in the case
considered in \cite{ACW}. See \cite[Lemma 5.1]{barkatou2016computing}. Theorem \ref{katz:algebra} completes the mathematical background of the algorithm, although it is technically not needed in it.

{\bf Organization of the paper.} In Section~\ref{section2}, we recall some notions on differential modules, tensor constructions and differential Galois theory. In Section~\ref{section3}, we prove our two theorems on the criteria for a linear differential system to be in reduced form. In Section~\ref{section4}, we apply the previous results to the study of the Katz algebra.

\section{Notation and definitions}\label{section2}

We consider a characteristic zero differential field $(k,\partial)$, that is a characteristic zero field $k$ with a derivation $\partial\colon k\to k$, such that $\partial(a+b)=\partial(a)+\partial(b)$ and
$\partial(ab)=\partial(a)b+a\partial(b)$, for any $a,b\in k$. We suppose that the subfield of constants $C:=k^\partial=\{f\in k\colon \partial f=0\}$ is algebraically closed.

\subsection{Differential modules}\label{section2.1}

A differential module $\cM=(M,\nabla)$ over $k$ (of rank $n$) is a $k$-vector space $M$ of dimension $n$, with a $C$-linear map $\na\colon M\to M$ such that $\na(fm)=\partial(f)m+f\nabla(m)$ for any $f\in k$ and any $m\in M$. For a detailed exposition on differential modules, see \cite[Section~2.2]{vdPutSingerDifferential}. We denote by~$M^\na$ or~$\ker \na$ the set of {\em horizontal elements} of~$\cM$, that is elements $m \in M$ satisfying $\na(m) =0$. This is a $C$-vector space of dimension at most~$n$.

{\bf Main properties of differential modules.}
Given a basis (denoted as a row) $\ul e :=(e_1, \dots, e_n)$ of $M$ over $k$,
the action of $\nabla$ with respect to the basis $\ul e$
is described by a square matrix $A\in M_n(k)$ as follows
\begin{gather*}
\na\ul e=-\ul e A.
\end{gather*}
For any $\vec y\in k^n$ such that $\ul e \vec y$ represents an element of $M$, we have
$\na(\ul e \vec y)=\ul e(\partial \vec y-A \vec y)$.
{Thus horizontal elements of $\cM$ correspond to solutions over $k$ of the differential system}
\begin{gather*}
[A]\colon \ \partial \vec y=A\vec y.
\end{gather*}
We say that $[A]\colon \partial \vec y=A\vec y$ is the linear differential system associated to $\cM$ with respect to the basis~$\ul e$.

If $\ul f=\ul e P$, with $P\in\GL_n(k)$, is another basis of $M$, then the horizontal elements of $\cM$ are of the form $\ul f \vec z$, with $\vec z\in k^n$, where $\vec{z}$ verifies the linear differential system
\begin{gather*}
\partial(\vec z)=P[A]\vec z,
\qquad
P[A]:=P^{-1}AP-P^{-1}\partial(P).
\end{gather*}
We say that two matrices $A, B \in M_n(k)$ are \emph{equivalent over~$k$} if there exists a \emph{gauge transformation} $P\in\GL_n(k)$ such that $B=P[A]$.

Notice that one can extend the scalars of $M$ to a field extension $k^\p$ of $k$, equipped with an extension of $\partial$. The Leibnitz rule allows to extend $\nabla$ to $M\otimes_k k^\p$, so that
it makes sense to consider gauge transformations in $\GL_n(k^\p)$, as counterpart of basis changes
of $\cM\otimes_k k^\p=(M\otimes_k k^\p,\nabla)$.

{\bf Algebraic constructions.} Let us start by formalizing what we mean by construction of linear algebra.
\begin{Definition}\label{defn:construction}
A \emph{construction of linear algebra} is a finite iteration of the basic constructors
$\oplus$ (direct sum), $\otimes$ (tensor product), $\ast$ (dual),
${\rm Sym}^r$ ($r$-th symmetric power, for $r\in\N$) and $\wedge^r$ ($r$-th exterior power).
Given a construction of linear algebra and a vector space~$M$, we denote by $\Constr(M)$ the
finite-dimensional $k$-vector space obtained by applying the construction to~$M$.
\end{Definition}

Given vector spaces $M_1$ and $M_2$ with respective bases $\ul{e}$ and $\ul{f}$,
an application of each of the above basic constructors produces \emph{canonically} a new basis:
for instance, $(e_1,\ldots,e_n, f_1,\ldots, f_m)$ is a basis for $M_1 \oplus M_2$,
$( e_i \otimes f_j \, | \, i=1,\ldots,n; j=1,\ldots,m )$
is a basis for $M_1\otimes M_2$, and so on (see \cite[Section~2.2, p.~42]{vdPutSingerDifferential}, and \cite[Section~3]{ACW}). This way, given a vector space $M$ with basis~$\ul{e}$ and a construction~$\Constr$, we iteratively construct a \emph{canonical} basis, denoted $\Constr(\ul{e})$, of $\Constr(M)$. Two different constructions may produce two isomorphic vector spaces. In this case, there is a canonical isomorphism between the two vector spaces which identifies the canonical bases.

\begin{Remark}We need to make some comments on our lists of algebraic constructions.
\begin{enumerate}\itemsep=0pt
 \item Notice that this list implicitly contains $\Hom$ (homomorphisms) and $\End$ (endomorphisms) via the canonical identifications $\Hom(M_1,M_2) \cong M_1 \otimes M_2^{\ast}$ and $\End(M)\cong M \otimes M^{\ast}$. See also Example~\ref{exa:EndM} below.
 \item Some authors use additional basic constructors such as quotients and subspaces. We choose not to do so to avoid the rising of apparent singularities (see \cite[Remark~18]{ACW}) and because this list is sufficient for our purposes. See Section~\ref{subsec:main-explicit}, where we need an ordinary point for all constructions.
\end{enumerate}
\end{Remark}

Let $H$ denote a linear algebraic group, with Lie algebra $\mathfrak{h}$, acting on $M$. Let the linear map $\sigma\in H$ have a matrix $U$ in the basis $\ul{e}$. The morphism induced by $\sigma$ on $\Constr(M)$ is denoted by $\Constr(\sigma)$. Its matrix in the basis $\Constr(\ul{e})$ is denoted by $\Constr(U)$ and the map $U\mapsto \Constr(U)$ is a group morphism. Similarly, for $h\in \mathfrak{h}$, it acts on $M$ as a linear derivation~$D_h$ with matrix~$N$. The action of this linear derivation~$D_h$ on the basis $\Constr(\ul{e})$ of $\Constr(M)$ induces a matrix~$\constr(N)$; the map $N\mapsto \constr(N)$ is a Lie algebra morphism. In what follows, we will use the gothical letters $\constr$ for these constructions ``in the sense of Lie algebras''.

Note that the entries of $\Constr(U)$ are \emph{polynomials} in the entries of $U$ and in $1/\det(U)$; the $1/\det(U)$ is needed to have duals as, when $\Constr(M)=M^{\ast}$, $\Constr(U)=\big(U^{-1}\big)^{\rm T}$. The entries of $\constr(N)$ are \emph{linear forms} in the entries of $N$.
See \cite[Section~2.4, p.~53]{vdPutSingerDifferential} and \cite[Sections~3.1 and~3.2]{ACW}.

Let $P\in\GL_n(k)$. If we consider a change of basis $\ul f=\ul e P$ in the vector space $M$, then the corresponding change-of-basis matrix in $\Constr(M)$ will be given by $\Constr(P)$ and we have $\Constr(\ul f)=\Constr(\ul e)\Constr(P)$.
The constructions of linear algebra apply functorially to differential modules.
Let $\cM=(M,\nabla)$ be a differential module over~$k$. The operator $\nabla$ induces a $C$-linear map from $\Constr(M)$ to $\Constr(M)$,
that we will also denote by $\nabla$,
defining a differential module structure over $\Constr(M)$. We will call the latter a \emph{construction of $\cM$} (or \emph{tensor construction}) and denote it
$\Constr(\cM)=(\Constr(M),\na)$.

Let $-A$ be the matrix of $\na$ with respect to a basis $\ul e$ as defined above.
For any construction $\Constr(\cM)$, the matrix of $\na$ with respect to $\Constr(\ul e)$ will be $-\constr(A)$. For example, in the case $\Constr(\cM) = \cM^*$, $\constr(A) = -A^{\rm T}$, while if $\Constr(\cM) = \cM\otimes_k\cM^\ast$, then we have $\constr(A) = A\otimes I_n -I_n\otimes A^{\rm T}$.
See \cite[Sections~3.1 and~3.2]{ACW}.

\begin{Example}\label{exa:EndM}
Let $\cM=(M,\nabla)$ be a differential module over $k$ and let $-A$ be the matrix of~$\na$ with respect to a~basis~$\ul e$. We consider the differential module $\End(\cM)=(\End_k(M),\na)$. If $\varphi\in\End_k(M)$ then $\na(\varphi)$ is the endomorphism of $M$ defined by $\nabla(\varphi)(m)=\nabla(\varphi(m))-\varphi(\nabla(m))$, for all $m\in M$.
If $F$ is the matrix of $\varphi$ with respect to the basis $\ul e$, one can check that we have
$\nabla(\varphi)(\ul e)=\ul e(\partial F-AF+FA)$. If we denote by {square matrices}
$F$ the elements of $\End(M)$ with respect to the basis induced by $\ul e$,
then the linear differential system associated to $\End(\cM)$ with respect to
the basis induced by $\ul e$ is $\partial F =AF-FA$.

The horizontal elements of $\End(\cM)$ are the elements $\varphi\in\End_k(M)$ such that we have $\nabla(\varphi(m))=\varphi(\nabla(m))$, $\forall\, m\in M.$ Thus the horizontal elements of the differential module $\End(\cM)$ are exactly the $k$-endomorphisms of $M$ which commute with~$\na$. They are called \emph{differential module endomorphisms} of~$\cM$. They form a~$C$-algebra denoted by $\cE{(\cM)}$ which is called the \emph{eigenring} of~$\cM$.

Among all the possible constructions, the differential module $\cM\otimes_k\cM^\ast$ will play a special role in the exposition below. If we identify it canonically to $\End(\cM)$, then the linear differential system associated to $\cM\otimes_k\cM^\ast$ in the basis induced by $\ul e$ is exactly
$\partial F =AF-FA$. The set $\cE([A])$ of matrices $F \in M_n(k)$ satisfying the above matrix differential equation is called the \emph{eigenring} of the system $[A]\colon \partial \vec y=A\vec y$. It is isomorphic (as a $C$-algebra) to $\cE(\cM)$.
\end{Example}

\subsection{Picard--Vessiot extensions}

We introduce very briefly some notions of differential Galois theory, with the main purpose of fixing the notation.
There exists several detailed introduction to the topic. We refer to \cite{vdPutSingerDifferential} for a general introduction
and to \cite{ACW} for more specific notions which are needed in this paper.

Let us consider the linear differential system $[A]\colon \partial \vec y=A\vec y$.
To any such system we can attach a $k$-algebra $R$, with an extension of $\partial$,
having the following properties:
\begin{enumerate}\itemsep=0pt
 \item[1)] there exists $U\in\GL_n(R)$ such that $\partial U=AU$;
 \item[2)] the entries of $U$ plus $\det U^{-1}$ generate $R$ over $k$, namely $R=k\big[U,\det U^{-1}\big]$;
 \item[3)] $R$ has no proper non-trivial ideals stable by $\partial$, i.e., it is a simple differential ring.
\end{enumerate}
We say that $R$ is a {\em Picard--Vessiot ring} of $k$ for $[A]$.
It is an integral domain and its ring of constants $R^\partial$ coincides with $C$. Its quotient field $K={\rm Frac}(R)$ is generated (as a field) by the entries of $U$ and its subfield of constant is again $C$. We call $K$ a {\em Picard--Vessiot extension} of $k$ for $[A]$.

\begin{Remark}\label{rmk:triviality} We are going to use several properties of Picard--Vessiot rings and extensions, namely:
\begin{enumerate}\itemsep=0pt
\item If $[A]$ and $[B]$ are two equivalent systems over $k$ then any Picard--Vessiot extension of $k$ for $[A]$ is a Picard--Vessiot extension of~$k$ for~$[B]$. Hence one can define a Picard--Vessiot extension of $k$ for a given differential module $\cM =(M,\na)$ as a Picard--Vessiot extension of~$k$ for the differential system associated to $\cM$ with respect to a basis of~$M$.
\item Let $K$ be a Picard--Vessiot extension of $k$ for differential module $\cM =(M,\na)$. The Leibnitz rule allows to endow $M\otimes_k K$ with a natural structure of differential module over~$K$, which will be denoted by $\cM\otimes_k K$. The definition of~$K$ implies that $\cM\otimes_k K$ is trivial, i.e., $\cM\otimes_k K$ admits a basis over~$K$ of horizontal elements. One can show that if a module is trivial, then all its algebraic constructions and their subquotients are trivial. See \cite[Exercice~2.12,5]{vdPutSingerDifferential}.
\item Let $V:=(\cM\otimes_k K)^\nabla$ be the $C$-vector space of the horizontal elements of $\cM\otimes_k K$. As already pointed out, it has dimension $n$.
\end{enumerate}
\end{Remark}

We give now a definition that we will use in the main theorem.

\begin{Definition}\label{defn:(semi-)invariants}
Let $\cM=(M,\nabla)$ be a differential module and $K$ a Picard--Vessiot extension. A {\em semi-invariant} of $\cM$ is a horizontal element $m\otimes g$ contained in
some construction of the form $\Constr(\cM\otimes_k K)\cong (\Constr(M)\otimes_k K,\na)$,
i.e., an element $m\otimes g$ with $m \in\Constr(M)$ and $g \in K$ such that $\nabla(m\otimes g)=0$.
If $m$ is a horizontal element in some construction $\Constr(M)$, then it is called an {\em invariant} of $\cM$.
\end{Definition}

For the convenience of the reader, we reprove the following classical lemma that we will use
in this work.

\begin{Lemma}\label{lemma:semi-invariants}
Let $\cM=(M,\nabla)$ be a differential module and $K$ a {\em Picard--Vessiot extension} of~$k$ for~$\cM$. Fix a basis~$\ul e$ of~$\cM$ and let $[A]\colon \partial \vec y=A \vec y$ be the associated linear differential system. The following statements are equivalent:
\begin{enumerate}\itemsep=0pt
\item[$1)$] there exists $m\in M$ such that $\nabla(m)= f m$, for some $ f \in k$ $($i.e., $m$ generates a $\nabla$-stable line$)$;
\item[$2)$] there exists a solution $g \vec v$ over $K$ of $\partial\vec y= A \vec y$,
with $g\in K$, such that $\partial(g)/g=-f \in k$, and $\vec v\in k^n$.
\end{enumerate}
\end{Lemma}

\begin{proof}Let us assume that there exists $m\in M$ such that $\nabla(m)= f m$, for some $f\in k$. Then the line $L$ generated by $m$ over $k$ is a differential module $\cL$, and $\cL\otimes_k K \subset \cM\otimes_k K $ is a trivial differential module (see Remark~\ref{rmk:triviality}). Thus there exists $g\in K$, $g\neq 0$, such that $m\otimes g$ is a~horizontal element of $\cL\otimes_k K$. It follows that we have
\begin{gather*}
0=\nabla(m\otimes g)=f m\otimes g+m\otimes \partial(g) = m\otimes ( fg+ \partial(g) ),
\end{gather*}
which implies $\partial(g)/g= -f \in k$. If we define $\vec v\in k^n$ by the relation $m= {\ul e}\vec v$, then we can check that $g\vec v$ is a solution vector of $\partial\vec y= A \vec y$.

Let us now suppose that we are in the situation described in the second assertion. If we define $m:={\ul e}\vec v$, then we have $\nabla(m)={\ul e}(\partial \vec v-A \vec v)$. Moreover, we have $\partial(g\vec v)=\partial(g) \vec v + g \partial \vec v=A g \vec v$ so that $\partial \vec v-A \vec v = -(\partial(g)/g) \vec v=f \vec v$. Finally $\nabla(m)=f m$, with $f \in k$ which ends the proof.
\end{proof}

\begin{Remark}\label{rmk:semi-invariants}
In particular, Lemma~\ref{lemma:semi-invariants} above implies that, if $m\otimes g$ is a semi-invariant,
then necessarily $\nabla(m)=(-\partial(g)/g)m$, or equivalently $m$ generates a $\nabla$-stable line $L$ contained in a~construction~$\Constr(M)$ of~$M$. If $m^\p=hm$, for some $h\in k$, then $\nabla(m^\p)=(-\partial(g)/g+\partial(h)/h)m^\p$ and
$\nabla(m^\p\otimes (g/h))=0$. Moreover, if $m\otimes g$ is a semi-invariant of $\cM$, such that $m$ generates a line $L$ in a~construction of~$M$, and $c\in C$ is a~nonzero constant, then $c(m\otimes g)$ is another semi-invariant, corresponding to the same line $L$. One can prove that all the semi-invariants can be obtained in this way.
Roughly speaking, semi-invariants of $\cM$ correspond to exponential solutions of $\partial\vec y= \constr(A) \vec y$ and invariants correspond to rational solutions.
For more details on these definitions, see \cite[Section~3.4]{ACW}.
\end{Remark}

For further reference we recall the following lemma:

\begin{Lemma}[{\cite[Lemma 29]{ACW}}]\label{lemma:invariants}
Let $[A]\colon \partial \vec y=A\vec y$ be the linear differential system associated to $\cM$
with respect to a fixed basis $\ul e$. We suppose that there exist
an algebraic extension $k^\p/k$ and a matrix $P\in\GL_n(k^\p)$ such that, for
any invariant of $\cM$ given by a horizontal element in some construction $\Constr(\cM)$ of coordinates $\vec v$ with respect to the basis $\Constr(\ul e)$, the vector
$\Constr(P)^{-1}\vec v$ has constant coordinates.
Then the same property holds for any invariant of $\cM\otimes_k \ol k$,
where $\ol k$ is the algebraic closure of $k$.
\end{Lemma}

\subsection{The differential Galois group}\label{sebsec:GaloisGroup}

The differential Galois group of $[A]$ (or, equivalently, of $\cM$)
is defined as
\[
 \Gal([A]) := \Aut^\partial(K/k),
\]
where $K$ is a Picard--Vessiot extension of $k$ for $[A]$ and
\begin{gather*}
 \Aut^\partial(K/k) := \big\{\varphi\colon K\to K,\text{~field automorphism s.t.~}\forall\, f\in k ,\\
 \hphantom{\Aut^\partial(K/k) := \{}{} \text{we have~} \varphi(f)=f \text{~and~}[\varphi,\partial]=0\big\}.
\end{gather*}
The first result of the differential Galois theory is that any fundamental matrix of solutions $U\in\GL_n(R)$ of $[A]$ determines a faithful representation of $\Gal([A])$ as a \emph{linear algebraic group} defined over $C$:
\begin{align*}
 \Gal([A])& \to \GL_n(C), \\
 \varphi & \mapsto U^{-1}\varphi(U).
\end{align*}
In fact, $\varphi(U)$ is a fundamental matrix of solutions of $[A]$, therefore $U^{-1}\varphi(U)$ must be an invertible matrix with constant coefficients. The choice of another fundamental matrix of solutions leads to a conjugated representation. We will sometimes simply call $G$ the differential Galois group $\Gal([A])$, identifying it with its image via the morphism above and without mentioning the matrix $U$, unless the context makes it necessary.

We are not explaining here any result on the Galois correspondence
and we refer the interested reader to
the literature.
For the purpose of this paper, we mostly need to know that, if $k^{\circ}$ is the relative algebraic closure of~$k$ in~$K$, then $K/k^{\circ}$ is a Picard--Vessiot extension for~$[A]$
over~$k^{\circ}$. Moreover the field of constants is still~$C$ and
we have $\Aut^\partial(K/k^{\circ})=G^\circ$, where $G^\circ$ is the connected component of~$G$ containing~$1$. This means that the differential Galois group of~$[A]$ over~$k^{\circ}$ coincides with the automorphisms of~$G$ that fix~$k^{\circ}$ and can be identified with~$G^\circ$.

{\bf The Galois--Lie algebra $\boldsymbol{\liegal([A])}$.}
Since the differential Galois group $G=\Gal([A])$ is an algebraic group over $C$, one can naturally consider its Lie algebra
$\g:=\liegal([A])$ called {\em Galois--Lie algebra}, that is the
tangent space to $G^\circ$ at $1$.
If we look at its $C$-rational points, we have{\samepage
\begin{gather*}
\liegal([A])(C)=\{N\in M_n(C)\colon 1+\veps N\in \Gal([A])(C[\veps])\},
\end{gather*}
where $\Gal([A])(C[\veps])$ are the rational points of $\Gal([A])$
over the $C$-algebra $C[\veps]$, with $\veps^2=0$.}

As it is an algebraic Lie algebra over $C$, $\liegal([A])$ is
generated as a $C$-vector space by a finite subset of $M_n(C)$.
It will be useful to notice that the same subset of matrices of $M_n(C)$ generates, as a $k$-vector space,
the algebra $\liegal([A])(k)$ of $k$-rational points of $\liegal([A])$.
For further reference we recall the following result:

\begin{Proposition}[{\cite[Proposition 1.31]{vdPutSingerDifferential}}]\label{prop:half-kolchin}
Let $\mathfrak{h}$ be an algebraic Lie algebra defined over~$C$ and such that
$A\in\mathfrak{h}(k)$. Then $\liegal([A])\subset \mathfrak{h}$.
\end{Proposition}

\section{Reduced forms}\label{section3}

\subsection{Characterization of reduced forms}

We keep the notation of the previous section.\label{section3.1}

\begin{Definition}We say that a linear differential system $[A]\colon \partial \vec y=A\vec y$ is in \emph{reduced form} when~$A$ belongs to $\liegal([A])(k)$. Let $\cM=(M,\na)$ be a differential module and $\ul e$ be a $k$-basis of $M$. We say that $\ul e$ is a~\emph{reduced basis} if the system associated to $\cM$ with respect to the basis $\ul e$ is in reduced form.
\end{Definition}

As pointed out in the introduction, Kolchin and Kovacic has proved that a reduced form always exists on a finite extension of~$k$. See \cite[Proposition~1.31]{vdPutSingerDifferential}. Criteria for reduced forms have been studied by Aparicio, Compoint and Weil in~\cite{ACW}. Their main result is generalized below. Our contribution in this theorem is the new characterization of reduced form stated in item~(3).

\begin{Theorem}\label{thm:main} Let $[A]\colon \partial \vec y=A\vec y$ be the linear differential system associated to a differential module $\cM=(M,\na)$ over~$k$, with respect to a fixed basis $\ul e$.
The following assertions are equivalent:
\begin{enumerate}\itemsep=0pt
 \item[$1)$] $[A]$ is in reduced form;
 \item[$2)$] for any construction $\Constr(\cM)$ of $\cM$, every $\nabla$-stable line
 of $\Constr(M)$ admits a constant basis $($i.e., a basis whose elements have constant coordinates with respect to the basis induced by~$\ul e)$;
 \item[$3)$] for any construction $\Constr(\cM)$ of~$\cM$, every $\nabla$-stable sub-$k$-vector space of $\Constr(M)$ admits a constant basis.
\end{enumerate}
If moreover $\cM$ is completely reducible $($i.e., it is direct sum of irreducibles$)$, then the assertions above are equivalent to
\begin{enumerate}\itemsep=0pt
\item[$4)$] any invariant of $\cM$ has constant coordinates.
\end{enumerate}
\end{Theorem}

\begin{proof}Lemma \ref{lemma:semi-invariants} and Remark~\ref{rmk:semi-invariants} show that the
constant bases of $\nabla$-stable lines correspond to the semi-invariants considered in~\cite{ACW}.
Therefore ``$(1)\Leftrightarrow (2)$'' is a reformulation of
\cite[Theorem~1]{ACW}. Since ``$(3)\Rightarrow (2)$'' is tautological, it is enough to prove that ``$(2)\Rightarrow (3)$''. Let $\cN=(N,\nabla)$ be a~construction of $\cM$ and $W$ be a~$\nabla$-stable sub-$k$-vector space of~$N$ of dimension~$d$. It follows that $\wedge^d W$ is $\nabla$-stable sub-$k$-vector space of dimension~$1$ of $\wedge^d N$. By assumption, there exists a~non-zero element of~$\wedge^d W$, whose coordinates $\vec w$ with respect to the basis induced by $\ul e$ on $\wedge^d N$ are in $C$.
Hence, in the basis induced by~$\ul e$, the map
\begin{align*}
\Psi\colon \ N&\to \wedge^{d+1} N,\\
 \vec v&\mapsto \vec w\wedge\vec v,
\end{align*}
is represented by a matrix with coefficients in $C$. It follows that $\ker \Psi$ has a basis of vectors with coordinates in $C$, with respect to the basis induced by $\ul e$. Since $\ker \Psi=W$, we have proved ``$(2)\Rightarrow (3)$''.

By definition, a horizontal element of a construction $\Constr(\cM)$ corresponds to a solution of
the associated differential system $\partial\vec y=\constr(A)\vec y$ with respect to
the basis $\Constr(\ul e)$.
Definition~\ref{defn:(semi-)invariants} and \cite[Exercice~2.38]{vdPutSingerDifferential} imply that the equivalence ``$(1)\Leftrightarrow(4)$'' coincides with
\cite[Proposition~27]{ACW} which ends the proof of Theorem~\ref{thm:main}.
\end{proof}

\begin{Remark}\label{rmk:constant-invariants}
Let $m\in\Constr(M)$ generate a line which is invariant by $\nabla$ and let $\ul e$ be a~reduced basis.
By Theorem~\ref{thm:main} above, we can choose $m$ such that $m=\Constr(\ul e)\vec v$, with $\vec v\in C^n$.
Moreover since $\nabla(m)=fm$ for some $f\in k$,
there exists $g\in K$ verifying $\partial(g)=fg$.
This means that $g\vec v$ is a solution vector of the system associated to $\Constr(\cM)$ with respect to
$\Constr(\ul e)$. This is what we mean when we say that invariants and semi-invariants have constant coordinates with respect to this reduced basis.
\end{Remark}

\begin{Example}\label{rmk:eigenring}Theorem~\ref{thm:main} has as corollary the following known fact:
in a reduced basis, the eigenring of $\cM$ (see Example~\ref{exa:EndM} for the definition) can be identified to the ring of matrices with coefficients in~$C$ which commute with~$A$, hence with all the elements of a Wei--Norman decomposition of~$A$ (see \cite[Section~2.2]{ACW} and~\cite{Wei-Norman}).
\end{Example}

The Kolchin--Kovacic theorem implies in particular that there exists a finite extension $k^\p/k$ such that
$\cM\otimes_k k^\p$ admits a reduced basis. See \cite[Proposition~1.31 and Corollary~1.32]{vdPutSingerDifferential} and \cite[Remark~31]{ACW}.

\begin{Definition}\label{reduction-field}
We say that $k^\p$ is a \emph{reduction field} when $\cM\otimes_k k^\p$ admits a reduced basis.
\end{Definition}

\begin{Remark}\label{rmk:kzero}
When $k^\p$ is a reduction field, then the differential Galois group of $\cM\otimes_k k^\p$ is connected, see \cite[Lemma~32]{ACW}, and
the Galois correspondence implies that $k^\p\cap K=k^\circ$, the fixed field of $G^{\circ}$ in the Picard--Vessiot extension $K$.
\end{Remark}

\subsection{Gauge transformation to a reduced form with local conditions}\label{subsec:main-explicit}

\begin{Assumption}\label{assumption:seidenberg}
In this section, we suppose that the field $k$ is a subfield of the field of meromorphic functions over a region $D$ of $\C$ in the variable $x$ such that
$x\in k$ and $\partial=\frac{{\rm d}}{{\rm d}x}$.
\end{Assumption}

\begin{Remark}
The field of rational functions $k=\C(x)$ satisfies the assumption above, as well as most differential fields occurring in the concrete examples.
Indeed let $k$ be any differential field.
Then the entries of the matrix $A$ of the linear differential system $[A]\colon \partial \vec y=A\vec y$, generate a differential field $\wtilde k$,
which is a finitely generated differential extension of $\Q$.
Seidenberg's embedding theorem (see \cite{seidprimel, seiden}) ensures that $\wtilde k$ can be embedded isomorphically in a differential field of meromorphic functions on an open region~$D$ of~$\C$.
\end{Remark}

We consider a linear differential system $[A]\colon \partial \vec y=A\vec y$, with coefficients in $k$.
For all points $x_0\in D$ such that $x_0$ is not a pole of $A$,
the system $[A]$ has a fundamental matrix of solu\-tions~$\what U_{x_0}(x)$ with the following properties:
\begin{enumerate}\itemsep=0pt
\item[1)] $\what U_{x_0}(x)\in\GL_n(\C[[x-x_0]])$,
\item[2)] $\what U_{x_0}(x_0)={\rm Id}$, the identity matrix.
\end{enumerate}
The columns of $\what U_{x_0}(x)$ generate over $\C$ the vector space $V$ of solutions of $[A]$ (contained in $\C[[x-x_0]]^n$),
and $K:=k(\what U_{x_0}(x))$ is a Picard--Vessiot extension of $[A]$.
Let $\Constr$ be a construction. Then $\Constr\big(\what U_{x_0}\big)(x_0)$ is the identity matrix (because $\Constr$ acts on matrices as a~group morphism). Furthermore, $\Constr\big(\what U_{x_0}\big)$ is a~fundamental matrix of solutions for~$[\constr(A)]$.
Consider an invariant of $\cM$ given by a horizontal element~$m$ in
$\Constr(\cM)$ such that~$m$ has coordinates
$\vec{v}(x)\in k^N$ with respect to the basis $\Constr(\ul e)$. The vector $\vec v(x)$ is a solution of $[\constr(A)]$ and hence we have $\vec v(x)=\Constr\big(\what U_{x_0}\big)(x)\vec w$ for some $\vec w$ with constant coefficients in~$\C$. The fact that $\Constr\big(\what U_{x_0}\big)(x_0)$ is the identity matrix
allows to conclude that $\vec w=\vec v(x_0)$.

In the theorem below, we give an algebraic characterization for a reduction matrix. The original Kolchin--Kovacic theorem ensures that, when the base field $k$ is a~$C^1$-field, there exists a~reduced basis defined on the relative algebraic closure~$k^\circ$ of~$k$ in the Picard--Vessiot extension~$K$. However, its proof is not effective and it relies on finding rational points on varieties and this part is not either, to our knowledge, algorithmic yet. In order to gain effectiveness, we enlarge~$k$ to an algebraic extension~$k^\p$, in which we can compute (without $C^1$-assumptions on $k$) a reduction matrix that is characterized by the property of transforming any invariant in its ``value at~$x_0$''. The field $k^\p$ may depend on the choice of~$x_0$. However, as seen in Remark~\ref{rmk:kzero}, we have $k^\p\cap K=k^\circ$ so $k^\p\cap K$ does not depend on the choice of~$x_0$.

The spirit of this result and of our proof appears in the proof of \cite[Theorem~3]{ACW} in a particular case. It is extended here to all constructions and this is useful for reduced form algorithms; see \cite[Lemma~5.1]{barkatou2016computing}, where this result was alluded to, without a full proof for lack of space.

\begin{Theorem}\label{thm:main-explicit} Let us consider a linear differential system $\partial\vec y=A\vec y$, defined over a field~$k$, associated to a completely reducible differential module $\cM$ with respect to a fixed basis $\ul e$. We choose a point $x_0\in D$ such that $A$ does not have a pole at $x_0$. Then, there exists a finite extension $k^\p$ of $k$ and a~matrix $P_{x_0}\in\GL_n(k^\p)$ such that $\ul{f}:=\ul{e} P_{x_0}$ is a reduced basis of $\cM\otimes_k k^\p$, having the following property:
 \begin{quote}
For any invariant of $\cM$ given by a horizontal element $m$ in some construction $\Constr(\cM)$ such that $m$ has coordinates $\vec{v}(x)\in k^N$ with respect to the basis $\Constr(\ul e)$, we have $\vec v(x)=\Constr(P_{x_0})\vec v(x_0)$.
 \end{quote}
\end{Theorem}

\begin{Remark}
In the last statement, we view $\Constr(\ul e)$ as a basis
of $\Constr(M\otimes_k k^\p)$ and identify an element $m\in\Constr(M)$ with its image $m\otimes 1\in\Constr(M)\otimes_k k^\p\cong \Constr(M\otimes_k k^\p)$.
\end{Remark}

\begin{proof}
In the new basis $\ul f$, if it exists, all invariants are constant therefore it is a reduced basis, thanks to Lemma~\ref{lemma:invariants} and Theorem \ref{thm:main}.
We now prove the existence.
Let $X$ be a matrix with indeterminate entries. For any invariant in some construction $\Constr(\cM)$ having coordinates $\vec{v}(x)\in k^N$ with respect to the basis $\Constr(\ul e)$, consider the equation $\vec v(x)=\Constr(X)\vec v(x_0)$. It provides an infinite set of polynomial equations over $k$ in the entries of $X$ which are all satisfied by the entries of a fundamental matrix of solutions $\what U_{x_0}(x)$ of $[A]$ (see the explanations above). As a consequence, the latter set of equations generates a proper ideal in the ring of polynomials in $n^2$ variables, which is finitely generated because of the noetherianity. The Nullstellensatz then ensures that there exists a solution $P\in\GL_n(k^\p)$, where $k^\p/k$ is a finite extension of $k$.
\end{proof}

\begin{Remark}\label{rmk:invariant-of-g}Assume that $\cM$ admits a reduced basis $\ul e$. Consider an invariant $m$ in some construction $\Constr(\cM)$ having (constant) coordinates $\vec{v}\in C^N$ with respect to the basis $\Constr(\ul e)$. Then we have $\vec{v} = \Constr\big({\what U_{x_0}}(x)\big) \vec{c}$ and the constant vector $\vec{c}$ is an invariant (in the usual sense of representation theory) of the Galois group $\Gal([A])$ in its representation induced by $\what U_{x_0}(x)$. Now, by construction, ${\what U}_{x_0}(x_0)={\rm Id}$ (the identity matrix); as $\Constr$ acts as a group morphism, we see that $\Constr\big({\what U}_{x_0}\big)(x_0)={\rm Id}$. Now, because $\vec{v}$ is constant, we have $\vec{v} = \Constr\big({\what U}_{x_0}\big)(x_0) \vec{c}$ so $\vec{v}=\vec{c}$. In a reduced basis, this observation allows to identify the invariants of~$\cM$ and those of~$\Gal([A])$.
\end{Remark}

\section{The intrinsic Galois--Lie algebra of Katz}\label{section4}

Let $\cM=(M,\nabla)$ be a differential module over $k$.

\begin{Definition}[\cite{Katzbull}]The \emph{intrinsic Galois group $\GK$ of Katz} is the set of $\varphi\in \GL(M)$ such that, for any $\nabla$-stable sub-$k$-vector space $N$ of a construction $\Constr(M)$, $N$ is set-wise stable under $\Constr(\varphi)$.

The \emph{Katz algebra} $\gK$ is the Lie algebra of $\GK$.
\end{Definition}

\begin{Remark}
For $\varphi\in \End(M)$, the functor $\mathfrak{constr}$ acts on $\varphi$ as a Lie algebra morphism, see the explanations before Example \ref{exa:EndM}. In particular, for any invariant $m\in \Constr(M)$ and $\varphi\in \gK$, we have $\mathfrak{constr}(\varphi)(m)=0$.
\end{Remark}

By Noetherianity, $\gK$ is a stabilizer of a finite family of differential modules contained in some constructions of $\cM$. This shows that $\gK$ is an \emph{algebraic} Lie algebra.
It is the central object of the famous Grothendieck--Katz conjecture on $p$-curvatures.
The following properties of $\gK$ are known, see \cite{Andregrothendieck,Katzbull,vdPutSingerDifferential}, but are reproved for self-containedness.

\begin{Lemma}\label{lemma:katz}Let $\cM=(M,\nabla)$ be a differential module over $k$, $K$ a Picard--Vessiot extension, $G$ the differential Galois group of $\cM$ and $\g$ the Galois--Lie algebra.
\begin{enumerate}\itemsep=0pt
\item[$1.$] The Lie algebra $\gK$ can be defined as the stabilizer of a single line in an algebraic construction of~$M$ $($which can be chosen to be $\nabla$-invariant$)$.
\item[$2.$] The Lie algebra $\gK$ is a differential module for the structure induced by $\cM\otimes_k\cM^*$.
\item[$3.$] $\gK=(\g \otimes_C K)^G$.
\end{enumerate}
\end{Lemma}

\begin{proof}1.~Because of the Noetherianity, a theorem of Chevalley ensures that
$\gK$ is the stabilizer of a finite family~$W_i$ of $\nabla$-stable $k$-vector spaces
contained in some constructions of~$M$.
A classical argument then shows that $\gK$ can be defined as the stabilizer of $\wedge^d\oplus_i W_i$, where~$d$ is the dimension of $\oplus_i W_i$ over~$k$. See \cite[proof of Proposition~9.3]{Katzbull}

2.~If $m$ generates a $\nabla$-stable line in some construction of~$M$, which defines $\gK$
as a stabilizer, and $\varphi\in\gK$, then
$\mathfrak{constr}(\varphi)(m)=\alpha m$ and $\nabla (m)=\beta m$, for some $\a,\be\in k$.
We conclude that
\begin{align*}
\nabla(\mathfrak{constr}(\varphi))(m)&=\nabla(\mathfrak{constr}(\varphi)(m)) -\mathfrak{constr}(\varphi)(\nabla(m))
 =\nabla(\alpha m)-\mathfrak{constr}(\varphi)(\beta m)\\
&=\partial(\alpha)m+\alpha \nabla(m)-\beta \mathfrak{constr}(\varphi)(m) =\partial(\alpha)m.
\end{align*}
Therefore $\nabla(\mathfrak{constr}(\varphi))$ belongs to $\gK$.

3.~The Tannakian correspondence will show that $\gK\cong\l(\g\otimes_k K\r)^G$. Indeed, we have seen that
$\gK$ is the stabilizer of a line $L$. Then $L_C:=(L\otimes_C K)^\nabla$ is a $C$-line defined in the
corresponding construction on the $C$-vector space of solutions $V:=(M\otimes_k K)^\nabla$.
It is stabilized by $\g$ so $\g \otimes_C K$ stabilizes $L_C \otimes_C K$.
Hence $(\g \otimes_C K)^G$ stabilises $(L_C \otimes_C K)^G$. Now $(L_C \otimes_C K)^G=L$ so
$(\g \otimes_C K)^G$ stabilises $L$ and thus $(\g \otimes_C K)^G \subseteq \gK$. In fact, the same argument shows that $(\g \otimes_C K)^G$ stabilizes any differential module in a construction and thus
$(\g \otimes_C K)^G = \gK$.

This ends the proof.
\end{proof}

\begin{Definition}\label{defn:Chevalley}
An invariant $m$ of $\cM$ in $\Constr(\cM)$ is called a \emph{Chevalley invariant} of $\gK$ when $\gK=\{ h \in \mathfrak{gl}(M) \, | \, h(m)=0 \}$.
\end{Definition}

\begin{Theorem} \label{katz:algebra} Let $\cM=(M,\nabla)$ be a differential module over $k$, $G$ the differential Galois group of $\cM$, $\g$ the Galois--Lie algebra and $\gK$ the Katz algebra. Let $k^\p$ be a reduction field so that $\cM\otimes_k k^\p$ admits a reduced basis $\ul e$.
If $\cM$ is completely reducible, then $\g\otimes_C k^\p = \gK \otimes_k k^\p$.
\end{Theorem}

\begin{proof} In a reduced basis of $\cM\otimes_k k'$, Remark~\ref{rmk:invariant-of-g} shows that $\g$ and $\gK\otimes_k k'$ have the same invariants. Now $\cM$ is completely reducible (equivalently, the Galois group $G$ is reductive) so both $\g$ and $\gK$ are completely reducible and this remains true after extension of scalars. From \cite[Chapter~II, Section~5.5, p.~92]{Bo91a} or \cite[Lemma~26]{ACW}, it follows that both $\g$ and $\gK\otimes_k k'$ are determined by their invariants.
These two observations show that a Chevalley invariant of~$\g$ is also a~Chevalley invariant of~$\gK\otimes_k k'$.
Let $m$ be a Chevalley invariant of $\gK\otimes_k k'$ in some construction $\Constr(M)\otimes_k k'$.
In the reduced basis, $m$ has constant coefficients: we have
$m=\Constr(\ul e)\vec v$, with $\vec v\in C^n$.
Thus, a matrix $N=(n_{ij})_{i,j}$ is in $\gK\otimes_k k'$ if and only if $\mathfrak{constr}(N) \vec v=0$. This relation yields a system of linear equations $\mathcal{L}( (n_{i,j})_{i,j})=0$
for the entries $(n_{i,j})_{i,j}$ of $N$, with coefficients in $C$.
Now, as seen above, $\vec v$ is also a Chevalley invariant of the Galois--Lie algebra $\g$.
So a matrix is in $\g$ if and only if its entries satisfy the same linear equations
$\mathcal{L}( (n_{i,j})_{i,j})=0$. Consequently, $\g\otimes_C k'$ and $\gK\otimes_k k'$ share the same constant basis so that $\g\otimes_C k'=\gK\otimes_k k'$.
\end{proof}

\begin{Definition} Let $k'$ be an extension of $k$ and $W$ be a subspace of a construction $\Constr(M)\allowbreak \otimes_k k'$. We say that $W$ is \emph{defined over $C$} when it is generated by constant matrices (i.e., when it is a $k'$-form of a $C$-space).
 \end{Definition}

\begin{Corollary}Let $k'$ be a reduction field so that $\cM\otimes_k k'$ admits a reduced basis $\ul e$. Then $\gK\otimes_k k'$ is defined over $C$. Moreover, if $A$ is the matrix of the associated linear differential system with respect to $\ul e$, then $A\in \gK\otimes_k k'$.
\end{Corollary}

\begin{proof} Let $L$ be a $\nabla$-invariant line in some $\Constr(M)$ such that $\gK$ is the stabilizer of $L$. By Theorem~\ref{thm:main}, we can choose a generator $m$ of {$L\otimes_C k'$} whose coordinates are constant with respect to the reduced basis. This shows that {$\gK\otimes_k k'$ (but not $\gK$ in general)} is defined over $C$. The fact that $A\in \gK\otimes_k k'$ follows from the definition of a reduced basis and from Theorem~\ref{katz:algebra}.
\end{proof}

\begin{Example}We illustrate the previous results on an example where everything can be checked by hand calculations. Let \[A = \left( \begin{matrix} 0&1\\ x&\frac{1}{2x}
\end{matrix} \right).
\]
The Galois group $G$ is a central extension of the infinite dihedral group (see \cite[Example~6.1]{ACW}). The connected component $G^\circ$ of $G$ containing $1$ is the multiplicative group $G_m$ with Lie algebra~$\mathfrak{g}_m$ generated by \[ \left( \begin{matrix} 1 &0 \\ 0&-1\end{matrix} \right). \] Furthermore, using~\cite{barkatou2016computing}, we can see that the Katz algebra is $1$-dimensional and it is generated by
\[ N_1:=\left( \begin{matrix} 0&\frac1x\\ 1&0
\end{matrix} \right).
\]
We have $\nabla(N_1):=N_1'-[A,N_1] = -\frac{1}{2x} N_1$. Over $k=C(x)$,
{no nonzero multiple of $N_1$
is conjugated to a constant matrix.
However, over the finite extension $k':=C(\sqrt{x})$, $\sqrt{x}N_1$ is conjugated to a constant matrix.}
Indeed,
we have $N_1 = P^{-1} D P$ with
\[ D = \left( \begin{matrix} {\frac {1}{\sqrt {x}}}&0
\\ 0&-\frac {1}{\sqrt {x}}\end{matrix} \right)
= \frac{1}{\sqrt{x}}\left(\begin{matrix} 1 &0 \\ 0&-1\end{matrix} \right),
\qquad
P= \left(\begin {matrix} 1&-1\\ \sqrt {x}&\sqrt{x}\end{matrix} \right).
\]

We see that $\gK\otimes_k k'$ is now generated by the (constant) generator of the Galois--Lie algebra~$\mathfrak{g}_m$ and, applying the gauge transformation $P$, we have the reduced form
\[ P[A]= \left(\begin{matrix} \sqrt {x}&0\\ 0&-\sqrt {x}\end{matrix} \right)
=\sqrt{x} \left(\begin{matrix} 1 &0 \\ 0&-1\end{matrix} \right).
\]
This shows that $\gK\otimes_k k'=\mathfrak{g}_m \otimes_C k'$ and $P[A] \in \gK\otimes_k k'$. Note
that $\gK$ is \emph{not} defined over $C$ whereas $\gK\otimes_k k'$ is.
\end{Example}
\begin{Remark}Following the definition of the eigenring $\mathcal{E}([A])$ in Example~\ref{exa:EndM},
any matrix in~$\mathcal{E}([A])$ maps a solution of~$[A]$ to a solution of~$[A]$.
Example~\ref{rmk:eigenring} showed that, in a reduced basis, the eigenring is generated by constant matrices which commute with~$A$. In reduced form, we have $A=\sum\limits_{i=1}^d f_i N_i$ where $f_i\in k'$, the $f_i$ are linearly independent over $C$ and the $N_i$ generate~$\mathfrak{g}$
(see \cite[Definition~6 and Remark~7]{ACW}). Theorem~\ref{katz:algebra} shows that the $N_i$ also generate $\gK\otimes_k k'$ as a Lie algebra. Given a constant matrix $T\in \mathcal{E}([A])$, we have $[A,T]=0$, i.e.,
$\sum\limits_{i=1}^d f_i [N_i,T]=0$.
As the $f_i$ are linearly independent over $C$, it follows that we have $[N_i,T]=0$ for all $i$.
So, in a reduced basis, the eigenring becomes the set of constant matrices which commute with $\gK\otimes_k k'$ (or, equivalently, with $\mathfrak{g}$).
\end{Remark}

In \cite{Andregrothendieck}, Andr\'e warns that spaces which are stable under $\gK$ may not be stable under $\nabla$. We can see this easily using reduced forms. Consider a linear differential system $[A]\colon \partial \vec y=A\vec y$ having all its solutions algebraic over~$k$. Then $\gK=0$ so that anything is stable under~$\gK$. In a reduced basis, the matrix of the linear differential system is the zero matrix and $\nabla$ coincides with $\partial=\frac{{\rm d}}{{\rm d}x}$. A random vector in~$k^n$ is not stable under $\nabla$ even though it is stable under~$\gK$. However, any line defined over~$C$ is clearly stable under~$\partial$ and hence under~$\nabla$. The next result builds on this observation to characterize which spaces, among those which are stable under~$\gK$, are stable under $\nabla$.

\begin{Proposition}\label{stable-sous-gk} Let $\cM=(M,\nabla)$ be a completely reducible differential module over~$k$ and~$k'$ denote a reduction field for $\cM$. Let $W$ be a subspace of a construction $\Constr(M)\otimes_k k'$. Then~$W$ is stable under $\nabla$ if and only if both conditions below are fulfilled:
\begin{enumerate}\itemsep=0pt
\item[$1)$] $W$ is stable under $\gK \otimes_k k'$;
\item[$2)$] $W$ is defined over $C$.
\end{enumerate}
\end{Proposition}

\begin{proof}{The ``only if'' part follows from the definition of the Katz algebra~(1) and from Theorem~\ref{thm:main}(2). Now, assume that $W$ is stable under $\gK \otimes_k k'$ and defined over~$C$.}
Let $C_1, \ldots, C_s$ denote a constant basis of $W$. Let $A$ be the matrix of the linear differential system associated with $\cM\otimes_k k'$ in a reduced basis. We have $A=\sum_j f_j(x) N_j$, where $f_j(x)\in k'$ and the~$N_j$ form a~(constant) basis of $\gK$ (and of $\g$). As the $C_i$ are constant, $\nabla$~acts on them via $\nabla(C_i)=-\mathfrak{constr}(A) C_i$ and thus $\nabla(C_i)= - \sum_j f_j(x) \mathfrak{constr}(N_j) C_i$. Now, by hypothesis $W$ is stable under $\gK \otimes_k k'$ so that $\mathfrak{constr}(N_j) C_i$ is a linear combination (over $k'$) of the $C_l$. It follows that $\nabla(C_i)$ is in $W$ as announced.
\end{proof}

Each space in $\Constr(V)$ which is invariant under $\mathfrak{g}$ is in (Tannakian) correspondence with a~submodule of $\Constr(M)$, invariant under $\nabla$ and hence under $\gK$. A reciprocal property would be to characterize, among all subspaces in any~\Constr(M) that are stable under~$\gK$, which ones are stable under $\nabla$ (and hence are in Tannakian correspondence with a $\mathfrak{g}$-module in $\Constr(V)$. Proposition~\ref{stable-sous-gk} shows that this it is possible to do that at the cost of extending scalars. When we study the subspaces of $\Constr(M)\otimes_k k'$ which are invariant under $\gK\otimes_k k'$, the ones which are invariant under $\nabla$ (and hence in tannakian correspondence with a $\mathfrak{g}$-module) are exactly those which admit a constant basis.

\subsection*{Acknowledgements}

We are grateful to the anonymous referees for their relevant suggestions which helped us to improve the clarity and quality of this work.

\pdfbookmark[1]{References}{ref}
\LastPageEnding


\begin{thebibliography}{99}
\footnotesize\itemsep=0pt

\bibitem{AmMiPo18a}
Amzallag E., Minchenko A., Pogudin G., Degree bound for toric envelope of a
 linear algebraic group, \href{https://arxiv.org/abs/1809.06489}{arXiv:1809.06489}.

\bibitem{Andregrothendieck}
Andr\'e Y., Sur la conjecture des {$p$}-courbures de {G}rothendieck--{K}atz
 et un probl\`eme de {D}work, in Geometric Aspects of {D}work Theory,
 {V}ols.~{I},~{II}, Walter de Gruyter, Berlin, 2004, 55--112.

\bibitem{ACW}
Aparicio-Monforte A., Compoint E., Weil J.-A., A characterization of reduced
 forms of linear differential systems, \href{https://doi.org/10.1016/j.jpaa.2012.11.007}{\textit{J.~Pure Appl. Algebra}} \textbf{217} (2013), 1504--1516, \href{https://arxiv.org/abs/1206.6661}{arXiv:1206.6661}.

\bibitem{barkatou2016computing}
Barkatou M., Cluzeau T., Weil J.-A., Di~Vizio L., Computing the {L}ie algebra of
 the differential {G}alois group of a linear differential system, in
 Proceedings of the 2016 {ACM} {I}nternational {S}ymposium on {S}ymbolic and
 {A}lgebraic {C}omputation, \href{https://doi.org/10.1145/2930889.2930932}{ACM}, New York, 2016, 63--70.

\bibitem{Be92c}
Bertrand D., Groupes alg\'ebriques et \'equations diff\'erentielles
 lin\'eaires, \textit{Ast\'erisque} \textbf{206} (1992), Exp.~No.~750, 4, 183--204.

\bibitem{Bo91a}
Borel A., Linear algebraic groups, 2nd~ed., \textit{Graduate Texts in Mathematics}, Vol.~126, \href{https://doi.org/10.1007/978-1-4612-0941-6}{Springer-Verlag}, New York, 1991.

\bibitem{CoSa18a}
Combot T., Sanabria C., A symplectic {K}ovacic's algorithm in dimension 4, in
 I{SSAC}'18~-- {P}roceedings of the 2018 {ACM} {I}nternational {S}ymposium on
 {S}ymbolic and {A}lgebraic {C}omputation, \href{https://doi.org/10.1145/3208976.3209005}{ACM}, New York, 2018, 143--150, \href{https://arxiv.org/abs/1802.01023}{arXiv:1802.01023}.

\bibitem{Compoint-Singer}
Compoint E., Singer M.F., Computing {G}alois groups of completely reducible
 differential equations, \href{https://doi.org/10.1006/jsco.1999.0311}{\textit{J.~Symbolic Comput.}} \textbf{28} (1999), 473--494.

\bibitem{DrWe19a}
Dreyfus T., Weil J.-A., Computing the {L}ie algebra of the differential {G}alois
 group: the reducible case, \href{https://arxiv.org/abs/1904.07925}{arXiv:1904.07925}.

\bibitem{Feng-hrushovski}
Feng R., Hrushovski's algorithm for computing the {G}alois group of a linear
 differential equation, \href{https://doi.org/10.1016/j.aam.2015.01.001}{\textit{Adv. in Appl. Math.}} \textbf{65} (2015), 1--37, \href{https://arxiv.org/abs/1312.5029}{arXiv:1312.5029}.

\bibitem{Hessinger}
Hessinger S.A., Computing the {G}alois group of a linear differential equation
 of order four, \href{https://doi.org/10.1007/s002000000055}{\textit{Appl. Algebra Engrg. Comm. Comput.}} \textbf{11} (2001), 489--536.

\bibitem{Hrushcomp}
Hrushovski E., Computing the {G}alois group of a linear differential equation,
 in Differential {G}alois Theory ({B}\c{e}dlewo, 2001), \textit{Banach Center
 Publ.}, Vol.~58, \href{https://doi.org/10.4064/bc58-0-9}{Polish Acad. Sci. Inst. Math.}, Warsaw, 2002, 97--138.

\bibitem{Katzbull}
Katz N.M., A conjecture in the arithmetic theory of differential equations,
 \href{https://doi.org/10.24033/bsmf.1960}{\textit{Bull. Soc. Math. France}} \textbf{110} (1982), 203--239.

\bibitem{Kovacic-algo}
Kovacic J.J., An algorithm for solving second order linear homogeneous
 differential equations, \href{https://doi.org/10.1016/S0747-7171(86)80010-4}{\textit{J.~Symbolic Comput.}} \textbf{2} (1986), 3--43.

\bibitem{NgPu10a}
Nguyen K.A., van~der Put M., Solving linear differential equations,
 \href{https://doi.org/10.4310/PAMQ.2010.v6.n1.a5}{\textit{Pure Appl. Math.~Q.}} \textbf{6} (2010), 173--208.

\bibitem{Person}
Person A.C., Solving homogeneous linear differential equations of order 4 in
 terms of equations of smaller order, Ph.D.~Thesis, {N}orth Carolina State
 University, 2002.

\bibitem{seidprimel}
Seidenberg A., Some basic theorems in differential algebra (characteristic
 {$p$}, arbitrary), \href{https://doi.org/10.2307/1990828}{\textit{Trans. Amer. Math. Soc.}} \textbf{73} (1952), 174--190.

\bibitem{seiden}
Seidenberg A., Contribution to the {P}icard--{V}essiot theory of homogeneous
 linear differential equations, \href{https://doi.org/10.2307/2372470}{\textit{Amer.~J. Math.}} \textbf{78} (1956), 808--818.

\bibitem{Singer-Ulmer-2nd3rdorder}
Singer M.F., Ulmer F., Galois groups of second and third order linear
 differential equations, \href{https://doi.org/10.1006/jsco.1993.1032}{\textit{J.~Symbolic Comput.}} \textbf{16} (1993), 9--36.

\bibitem{Ho07b}
van~der Hoeven J., Around the numeric-symbolic computation of differential
 {G}alois groups, \href{https://doi.org/10.1016/j.jsc.2006.03.007}{\textit{J.~Symbolic Comput.}} \textbf{42} (2007), 236--264.

\bibitem{vdPutSingerDifferential}
van~der Put M., Singer M.F., Galois theory of linear differential equations,
 \textit{Grundlehren der Mathematischen Wissenschaften}, Vol.~328,
 \href{https://doi.org/10.1007/978-3-642-55750-7}{Springer-Verlag}, Berlin, 2003.

\bibitem{vanheoij-4thorder}
van Hoeij M., Decomposing a 4th order linear differential equation as a
 symmetric product, in Differential {G}alois Theory ({B}\c{e}dlewo, 2001),
 \textit{Banach Center Publ.}, Vol.~58, \href{https://doi.org/10.4064/bc58-0-8}{Polish Acad. Sci. Inst. Math.}, Warsaw, 2002, 89--96.

\bibitem{HoRaUlWe99a}
van Hoeij M., Ragot J.F., Ulmer F., Weil J.-A., Liouvillian solutions of linear
 differential equations of order three and higher, \href{https://doi.org/10.1006/jsco.1999.0316}{\textit{J.~Symbolic Comput.}} \textbf{28} (1999), 589--609.

\bibitem{Wei-Norman}
Wei J., Norman E., On global representations of the solutions of linear
 differential equations as a product of exponentials, \href{https://doi.org/10.2307/2034065}{\textit{Proc. Amer. Math. Soc.}} \textbf{15} (1964), 327--334.

\end{thebibliography}
\end{document}